\documentclass{amsart}
\usepackage{amssymb}
\usepackage{amsmath}

\usepackage{stmaryrd}

\DeclareMathOperator{\Lip}{Lip}

\newcommand{\pilip}{\ensuremath{\pi^{\mathrm{Lip}}}}

\newcommand{\N}{\ensuremath{\mathbb{N}}}
\newcommand{\R}{\ensuremath{\mathbb{R}}}

\newcommand{\Z}{\ensuremath{\mathbb{Z}}}
\renewcommand{\H}{\ensuremath{\mathbb{H}}}

\newcommand{\haus}{\ensuremath{\mathcal{H}}}

\newcommand{\length}{\ensuremath{\operatorname{length}}}
\newcommand{\md}{\ensuremath{\operatorname{md}}}

\newtheorem{thm}{Theorem}
\newtheorem{question}{Question}
\newtheorem{theorem}[thm]{Theorem}

\newtheorem{lemma}[thm]{Lemma}

\newtheorem{cor}[thm]{Corollary}

\title{Lipschitz homotopy groups of the Heisenberg groups}
\author{Stefan Wenger}
\address{D\'{e}partement de math\'{e}matiques\\
 Universit\'{e} de Fribourg\\
 Chemin du Mus\'{e}e 23\\
 CH-1700 Fribourg\\
 Switzerland}
\email{stefan.wenger@unifr.ch}

\author{Robert Young}
\address{Department of Mathematics\\
  University of Toronto\\
  40 St.\ George St., Room 6290\\
  Toronto, Ontario  M5S 2E4\\
  Canada}
\email{ryoung@math.toronto.edu}
\thanks{The second author was supported by a Discovery Grant from the
  Natural Sciences and Engineering Research Council of Canada and a
  grant from the Connaught Fund, University of Toronto.}

\begin{document}
\bibliographystyle{amsalpha}

\begin{abstract}
  Lipschitz and horizontal maps from an $n$-dimensional space into the $(2n+1)$-dimensional
  Heisenberg group $\H^n$ are abundant, while maps from
  higher-dimensional spaces are much more restricted.
  DeJarnette-Haj{\l}asz-Lukyanenko-Tyson constructed horizontal maps from $S^k$
  to $\H^n$ which factor through $n$-spheres and showed that these
  maps have no smooth horizontal fillings.  In this paper, however, we
  build on an example of Kaufman to show that these maps sometimes
  have Lipschitz fillings.  This shows that the Lipschitz and the
  smooth horizontal homotopy groups of a space may differ.
  Conversely, we show that any Lipschitz map $S^k\to \H^1$ factors
  through a tree and is thus Lipschitz null-homotopic if $k\ge 2$.
\end{abstract}

\maketitle

DeJarnette, Haj{\l}asz, Lukyanenko, and Tyson recently initiated a study of
 smooth horizontal homotopy groups $\pi_k^H(X)$ and 
Lipschitz homotopy groups $\pilip_k(X)$ when $X$ is a sub-Riemannian manifold \cite{DHLT}.
By definition, $\pi_k^H(X)$ (and $\pilip_k(X)$) consist of
classes of smooth horizontal (respectively Lipschitz) maps $S^k\to X$, where two
maps lie in the same class if there is a homotopy $S^k\times [0,1]\to
X$ between them which is also smooth horizontal (resp.\ Lipschitz).

The groups $\pi_k^H(X)$ and $\pilip_k(X)$ capture more of the geometry of
sub-Riemannian manifolds than the usual homotopy groups $\pi_k(X)$.  For example,
if $X=\H^n$ is the $n$th Heisenberg group with its standard
Carnot-Carat\'eodory metric, it is homeomorphic to $\R^{2n+1}$, so its 
homotopy groups $\pi_k(\H^n)$ are trivial. Lipschitz maps to $\H^n$, however, are
more complicated.  If $f:D^{k}\to \H^n$ is Lipschitz, it must be a.e.\
Pansu
differentiable \cite{Pansu}.  In particular, the rank of $Df$
is a.e.\ at most $n$.  
Ambrosio-Kirchheim \cite{Ambrosio-Kirchheim-rectifiable}
and Magnani \cite{Magnani-unrectifiable} showed that, as a consequence, if $k>n$, then
$\mathcal{H}^k_{cc}(f(D^k))=0$.  Therefore, if $\alpha:S^n\to \H^n$ is
a smooth horizontal (and thus Lipschitz) embedding, it cannot be
extended to a Lipschitz map of a ball \cite{GromovCC, Balogh-Faessler, RigotWenger},
so $\pi_n^H(\H^n)$ and $\pilip_n(\H^n)$ are non-trivial.  In fact, these
groups are uncountably generated \cite{DHLT}.

The behavior of $\pi^H_k(\H^n)$ and $\pilip_k(\H^n)$ when $k>n$ is
just starting to be explored.  DeJarnette, Haj{\l}asz, Lukyanenko, and Tyson \cite{DHLT}
showed that if $\beta\in \pi_k(S^n)$ is nontrivial, then $\alpha\circ \beta:S^k\to \H^n$ is a
nontrivial element of $\pi_k^H(\H^n)$ (their theorem is stated for a
particular smooth embedding $\alpha$, but their methods generalize to arbitrary
smooth embeddings).  Their proof relies on
Sard's theorem, however, so it does not generalize to Lipschitz maps.  They
asked:
\begin{question}[{\cite[4.6]{DHLT}}]
  Is the map $\pi_k^H(\H^n)\to \pilip_k(\H^n)$ an isomorphism?
\end{question}
\begin{question}[{\cite[4.17]{DHLT}}]
  If $\alpha:S^n\to \H^n$ is a bilipschitz embedding, is the induced
  map $\pilip_k(S^n)\to \pilip_k(\H^n)$ an injection?
\end{question}

In this paper, we will show that even if $\beta\in \pi_k(S^n)$ is
nontrivial, $\alpha\circ \beta$ may be Lipschitz-null homotopic,
answering both of these questions in the negative. More precisely, we prove the following theorems.

\begin{thm}\label{thm:heisHigher}
  If $\alpha:S^n\to \H^n$ and $\beta:S^k\to S^n$ are Lipschitz maps
  and $n+2\le k<2n-1$, then $\alpha\circ \beta $ can be extended to a
  Lipschitz map $D^{k+1}\to \H^n$.
\end{thm}

Since this extension is Lipschitz, it is almost everywhere Pansu differentiable,
and the Pansu differential has rank $\le n$ wherever it is defined.
Another version of our construction proves:
\begin{thm}\label{thm:euclHigher}
  If $n+1\le k<2n-1$, then
  any Lipschitz map $\beta:S^k\to S^n$ can be extended to a Lipschitz
  map $D^{k+1}\to \R^{n+1}$ whose derivative has rank $\le n$ almost everywhere.
\end{thm}
Our constructions build on Kaufman's construction of a Lipschitz surjection from the unit
cube to the unit square whose derivative has rank 1 almost everywhere
\cite{KaufmanSingular}.

In the time since the writing of this paper, several related
results have appeared.  Haj{\l}asz, Schikorra, and Tyson used a
generalization of the Hopf invariant to prove that
$\pilip_{4d-1}(\H^{2d})$ is nontrivial \cite{HaScTy}.  Indeed, they
show that, if
$$\beta:S^{4d-1}\to S^{2d}\subset \R^{2d+1}$$
is a Lipschitz map with nonzero Hopf invariant, then any extension of
$\beta$ to a Lipschitz map $D^{4d}\to \R^{2d+1}$ must have rank $2d+1$
on a set of positive measure.  If $\alpha:S^{2d}\to \H^{2d}$ is a
Lipschitz embedding, then $\alpha\circ \beta$ is a nontrivial element
of $\pilip_{4d-1}(\H^{2d})$.  This implies that
Theorems~\ref{thm:heisHigher} and \ref{thm:euclHigher} need not hold
when $k=2n-1$.  Note that the theorems may still hold when $k\ge
2n-1$.  In Section~\ref{sec:constExt}, we will prove a generalization (see Theorem~\ref{thm:heisHilton} below)
of Theorem~\ref{thm:heisHigher} that holds, for instance, when $k\ge
n+2$ and $\beta$ is a suspension of a map $S^{k-1}\to S^{n-1}$.

Guth \cite{Guth} has also published results which complement the
results in this paper.  He considers the $d$-dilation of $C^1$ maps.
The $d$-dilation of a map bounds the amount that the map stretches
$d$-dimensional surfaces in its domain.  If $d=1$, this is the
Lipschitz constant of the map, and if the $d$-dilation of a map is
zero, its derivative has rank $< d$ at every point.  Among other
results, Guth shows that if $m>n$ and $\beta:S^{m}\to S^{n}$, is the
suspension of a map $S^{m-1}\to S^{n-1}$, then there is a $C^1$ map
$a'$ homotopic to $a$ such that $a'$ has arbitrarily small
$n$-dilation \cite[Prop.~1.1]{Guth}.

When $a:S^m\to S^n$ is a double suspension, we can construct a similar
map.  In fact, if $a$ is a double suspension, then there is a
Lipschitz map $a'$ homotopic to $a$ such that $a'$ has rank $n-1$.  By
Theorem~\ref{thm:euclHilton} below, if $\beta:S^{m-1}\to S^{n-1}$ is a
suspension, there is a Lipschitz extension $\gamma:D^{m}\to D^{n}$
with rank $n-1$.  By gluing two copies of $\gamma$ together, we obtain
a Lipschitz map $a':S^m\to S^n$ which has rank $n-1$ and is homotopic
to the suspension of $\beta$.  In Section~13.2 of \cite{Guth}, Guth
asks whether any such maps exist; this answers his question
positively.  Guth also proves a number of other results about the
existence and nonexistence of topologically non-trivial maps with low
$d$-dilation, and whether these other results can also be modified to
give maps with low rank is an open question.

In view of the results above it is natural to ask whether $\pilip_k(\H^n)$ is trivial when
$n+2\le k<2n-1$.  This may be hard to answer, since general Lipschitz
$k$-spheres in $\H^n$ may be more complicated.  While the spheres we
consider in Theorems~\ref{thm:heisHigher} and \ref{thm:euclHigher}
have image with Hausdorff dimension $n$, the methods we use to prove
the theorems can be adapted to produce Lipschitz maps of $k$-spheres
to $\H^n$ whose image has Hausdorff dimension arbitrarily close to
$k$.

When $n=1$, however, things are much simpler.  We will show:
\begin{thm}\label{thm:sphere-factor-tree}
  If $k\ge 2$, then any Lipschitz map $f:S^k\to \H^1$ factors through a metric tree.   
  That is, there is a metric tree $Z$ and there are Lipschitz maps $\psi: S^k\to Z$ and $\varphi: Z\to \H^1$ such that $f = \varphi\circ\psi$. 
\end{thm}
Recall that a metric tree or $\R$-tree is a geodesic metric space such
that every geodesic triangle is isometric to a tripod.  Note that
these trees may still have large images; for instance, Haj{\l}asz and Tyson \cite{HajTysCarnotSurjections} have adapted Kaufman's construction
\cite{KaufmanSingular} to produce a
$C^1$ horizontal surjection $\mathbb{R}^5\to \H^1$.  As a consequence of
Theorem~\ref{thm:sphere-factor-tree} we obtain:
\begin{cor}\label{cor:sphere-factor-tree}
  If $k\ge 2$ and $\alpha:S^k\to \H^1$, then $\alpha$ is Lipschitz
  null-homotopic.  Furthermore, for any $\epsilon>0$, $\alpha$ is
  $\epsilon$-close to a map whose image has Hausdorff dimension $1$.
\end{cor}
\begin{proof}
  For the first statement, since $Z$ is a metric tree, it is contractible by a Lipschitz
  homotopy $h:Z\times [0,1]\to Z$.  Composing this with $\psi$ and
  $\varphi$ gives a Lipschitz homotopy contracting $\alpha$ to a
  point.

  For the second statement, let $\lambda=\Lip(\alpha)$ and let $E$ be a
  finite $\epsilon/\lambda$ net of points in $S^k$.  Let $T$ be the
  convex hull of $\psi(E)$ in $Z$; this is a finite tree.  The
  closest-point projection $p:Z\to T$ is Lipschitz and moves each
  point of $\psi(S^k)$ a distance at most $\epsilon$, so $\varphi\circ p \circ \psi$
  is a Lipschitz map which is $\epsilon$-close to $\alpha$.  Its image
  is $\varphi(T)$, which has Hausdorff dimension $1$.
\end{proof}

Consequently, $\pilip_k(\H^1)=\{0\}$ for all $k\geq 2$.  In general, a
Lipschitz map $\alpha:X\to \H^n$ need not be $\epsilon$-close to a map
whose image has Hausdorff dimension $n$; the homotopies
constructed in Theorem~\ref{thm:heisHigher} cannot be approximated by
such maps.

Theorem~\ref{thm:sphere-factor-tree} is a special case of the
following theorem.  Recall that a metric space $(X, d)$ is said to be
quasi-convex if there exists $C$ such that any two points $x,x'\in X$
can be joined by a curve of length at most $Cd(x,x')$. Furthermore, a
metric space $(Y,d)$ is called purely $k$-unrectifiable if
$\haus^k(\varrho(C)) = 0$ for every Lipschitz map $\varrho$ from a
Borel subset $C\subset \R^k$ to $Y$. It
was shown in \cite{Ambrosio-Kirchheim-rectifiable,
  Magnani-unrectifiable} that the Heisenberg group $\H^n$, endowed
with a Carnot-Carath\'eodory metric, is purely $k$-unrectifiable for
$k\geq n+1$.

\begin{thm}\label{thm:factor-tree}
 Let $X$ be a quasi-convex metric space with $\pilip_1(X)=0$. Let furthermore $Y$ be a purely $2$-unrectifiable metric space. Then every Lipschitz map from $X$ to $Y$ factors through a metric tree. 
\end{thm}

Theorem~\ref{thm:factor-tree} will be proved in Section~\ref{section:factor-tree}. If $C$ is the quasi-convexity constant and $\psi$ and $\varphi$ are as above, then $\psi$ can be chosen to be $C\Lip(f)$-Lipschitz and $\varphi$ to be $1$-Lipschitz.  As a corollary, we find that $\pilip_k(Y) = \{0\}$ for all $k\geq 2$ and every purely $2$-unrectifiable space $Y$.

\noindent\emph{Acknowledgments:} We would like to thank Piotr Haj{\l}asz, Larry Guth, and the anonymous referee for several helpful discussions and comments.

\section{Preliminaries}
In this section we briefly collect some of the basic definitions and properties of the Heisenberg groups. We furthermore recall the necessary definitions of metric derivatives in metric spaces which will be needed for the proof of Theorem~\ref{thm:factor-tree}.

\subsection{Heisenberg groups}
The $n$th Heisenberg group $\H^n$, where $n\geq 1$, is the Lie group
given by $\H^n:= \R^{2n+1}= \R^n\times\R^n\times\R$ endowed with the
group multiplication
$$(x,y,z)\odot (x',y',z') = \left(x+x', y+y', z+z'+ \langle y,x'\rangle\right),$$
where $\langle \cdot,\cdot\rangle$ is the standard inner product on
$\R^n$. 
 A basis of left invariant vector fields on $\H^n$ is defined by 
 $$X_j = \frac{\partial}{\partial x_j} + y_j\frac{\partial}{\partial t}\quad\text{and}\quad Y_j = \frac{\partial}{\partial y_j}, \quad j=1, \dots, n,$$
 and $Z= \frac{\partial}{\partial z}$. The subbundle $H\H^n\subset T\H^n$ generated by the vector fields $X_j, Y_j, j=1,\dots, n$, is called the horizontal subbundle. A $C^1$-smooth map $f: M\to \H^n$, where $M$ is a smooth manifold, is called horizontal if the derivative $df$ of $f$ maps $TM$ to the horizontal sub bundle $H\H^n$.

 (There are many equivalent ways to define the Heisenberg group and
 its horizontal subbundle.  For example, for any symplectic form
 $\omega$ on $\R^{2n}$, one may define $\H^n:= \R^{2n+1}=
 \R^{2n}\times\R$ and let
 $$(v,z)\odot (v',z') = (v+v', z+z'+ \omega(v,v')).$$
 The group defined this way is isomorphic to the group defined above,
 and if $H'$ is any left-invariant subbundle of $T\H^n$ which is
 complementary to the bundle $Z= \frac{\partial}{\partial z}$, there
 is an isomorphism which takes $H'$ to $H\H^n$.)

The Heisenberg group $\H^n$ is naturally equipped with a family $(s_r)_{r>0}$ of dilation homomorphisms $s_r:\H^n\to\H^n$ defined by $$s_r(x,y,z):= (rx,ry, r^2z).$$

 Let $g_0$ be the left-invariant Riemannian metric on $\H^n$ such that the $X_j, Y_k, Z$ are pointwise orthonormal. The Carnot-Carath\'eodory metric on $\H^n$ corresponding to $g_0$ is defined by
\begin{equation*}
 d(x,y):= \inf\{\length_{g_0}(c): \text{ $c$ is a horizontal $C^1$ curve from $x$ to $y$}\},
\end{equation*}
where $\length_{g_0}(c)$ denotes the length of $c$ with respect to $g_0$. The metric $d$ on $\H^n$ is $1$-homogeneous with respect to the dilations $s_r$, that is, 
 $$d(s_r(w), s_r(w')) = rd(w,w')$$ for all $w,w'\in\H^n$. Throughout this paper, $\H^n$ will always be equipped with the Carnot-Carath\'eodory metric $d$ defined above or any metric which is biLipschitz equivalent to $d$.

\subsection{Metric derivatives}
We recall the definition of the metric derivative of a Lipschitz map from a Euclidean to a metric space, as introduced and studied by Kirchheim in \cite{Kirchheim-rectifiable}. For this, let $(X,d)$ be a metric space and $f:U\to X$ a Lipschitz map, where $U\subset\R^n$ is open. The metric derivative of $f$ at $x\in U$ in direction $v\in\R^n$ is defined by
\begin{equation*}
 \md f_x(v):= \lim_{r\to 0^+}\frac{d(f(x+rv), f(x))}{r}
\end{equation*}
if the limit exists. It was shown in \cite{Kirchheim-rectifiable,
  Ambrosio-Kirchheim-rectifiable} that for almost every $x\in U$ the
metric derivative $\md f_x(v)$ exists for all $v\in\R^n$ and defines a
semi-norm on $\R^n$. It can be shown (see for instance \cite[Thm.~2.7.6]{BurBurIva-Course}) that for any Lipschitz curve $c:[a,b]\to X$ we have $$\length(c) = \int_a^b\md c_t(1) dt.$$

\section{Constructing extensions}\label{sec:constExt}

In this section we prove Theorems~\ref{thm:heisHigher} and
\ref{thm:euclHigher}.  The restriction in these theorems that $k<2n-1$ can be weakened somewhat.  To state the
theorems in full generality, we will need to recall some facts about
the homotopy groups of wedges of spheres.

If $X=S^n\vee S^n$ is a wedge of $n$-spheres, let $\iota_i:S^n\to X$, $i=1,2$,
be the map into the $i$th factor of $X$.  In what follows, addition
will be taken in $\pi_k(S^n)$ or $\pi_k(X)$, so $\iota_1+\iota_2$ represents a sphere which wraps once around each factor of
the wedge product.

If $\beta:S^k\to S^n$, then  $(\iota_1+\iota_2) \circ \beta$ and
$\iota_1 \circ \beta+\iota_2 \circ \beta$ are not homotopic in general.  The
Hilton-Milnor theorem describes the difference between these two maps:
\begin{thm}[{cf.\ \cite[Thm.\ 8.3]{WhiteheadElts}}]
  If $n\ge 2$, there is an isomorphism
  $$\pi_{k}(X)\cong \pi_k(S^{n})\oplus \pi_k(S^{n}) \oplus \bigoplus_{j=0}^\infty \pi_k(S^{q_j(n-1)+1}),$$
  with $q_j$ a sequence of integers going to $\infty$ with $q_j\ge 2$.

  There are homomorphisms $h_j:\pi_k(S^n)\to \pi_k(S^{q_j(n-1)+1})$,
  $j=0,1,2,\dots$, such that if $\beta\in \pi_{k}(S^n)$, then
  $$(\iota_1+ \iota_2) \circ \beta=(\iota_1\circ \beta)+(\iota_2\circ \beta)+\sum_{j=0}^\infty w_j\circ h_j(\beta)$$ where 
  $$w_j:S^{q_j(n-1)+1}\to X$$
  is an iterated Whitehead product of $\iota_1$ and $\iota_2$ with $q_j$ terms.
\end{thm}

These $h_j(\beta)$'s are invariants of $\beta$ known as the \emph{Hopf-Hilton
  invariants}.  By the theorem, the Hopf-Hilton invariants vanish if
and only if
\begin{equation}\label{eq:hopf-hilton-vanish-equiv}
 (\iota_1+\iota_2) \circ \beta = \iota_1 \circ\beta+\iota_2\circ\beta.
 \end{equation}
In fact, the Hopf-Hilton invariants are the obstruction to the distributive law
holding for compositions of the form $(\alpha_1+\alpha_2)\circ
\beta$.  For any based space $Y$ and any based maps
$\alpha_1,\alpha_2:S^n\to Y$, we can define a map $\alpha:S^n\vee
S^n\to Y$ which is $\alpha_1$ on one wedge factor and $\alpha_2$ on
the other.  If \eqref{eq:hopf-hilton-vanish-equiv} holds, then 
\begin{align*}
  (\alpha_1+\alpha_2)\circ \beta&=\alpha \circ (\iota_1 + \iota_2)\circ\beta\\ 
  &=\alpha \circ (\iota_1 \circ\beta+ \iota_2\circ\beta)\\ 
  &=\alpha \circ \iota_1 \circ\beta+\alpha \circ \iota_2\circ\beta\\ 
  &=\alpha_1\circ\beta+\alpha_2\circ\beta,
\end{align*}
so the distributive law holds for $\beta$.  

The main application of the Hopf-Hilton invariants in this paper
involves maps to $m$-fold wedges of spheres.  Let $Y=S^n\vee\dots\vee
S^n$ be an $m$-fold wedge of $n$-spheres and let $i_j:S^n\to Y$, $j=1,\dots,
m$ be the map into the $j$th factor of $Y$.  If $\beta\in
\pi_k(S^n)$ is a map whose Hopf-Hilton invariants vanish, then
\begin{equation}
  \tag{$*$}\biggl(\sum_{j=1}^m i_j\biggr) \circ \beta = \sum_{j=1}^m
(i_j \circ \beta).\label{eq:HH}
\end{equation}
That is, the map that first sends $S^k$ to $S^n$ by $\beta$, then
wraps $S^n$ around all the wedge factors of $Y$ can be homotoped into
a sum of maps, each with image lying in a single wedge factor.

A notable case where the Hopf-Hilton invariants of $\beta$ vanish is when
$\beta\in \pi_k(S^n)$ is a
suspension of a map $\beta_0\in \pi_{k-1}(S^{n-1})$.  This is easiest to
see if we write $\beta$ in ``cylindrical'' coordinates.  That is, we
write points of $S^n$ as points in $S^{n-1}\times [0,1]$, identifying
$S^{n-1}\times \{0\}$ and $S^{n-1}\times \{1\}$ with the poles of
$S^n$.  In these coordinates, we can write $\beta(x,t)=(\beta_0(x),t)$.
Note that this map takes the northern and southern hemispheres of
$S^k$ to the northern and southern hemispheres of $S^n$, respectively.  

We can likewise view $S^n\vee S^n$ in ``cylindrical'' coordinates by
identifying it with $S^{n-1}\times [0,1]$.  In this case, we identify
$S^{n-1}\times \{0\}$ and $S^{n-1}\times \{1\}$ with the poles of
$S^n$ and identify the equator $S^{n-1}\times \{1/2\}$ with the
basepoint of the wedge.  Then we can represent $\iota=\iota_1 + 
\iota_2$ by the quotient map $S^n\to S^n\vee S^n$ which
collapses the equator of $S^n$ to the basepoint of the wedge and
write 
$$(\iota\circ \beta)(x,t)=(\beta_0(x),t)\in S^n\vee S^n.$$

The map $\iota\circ \beta$ wraps the northern hemisphere of $S^k$
around one of the factors of $S^n\vee S^n$, sends the equator to the
basepoint of the wedge, and wraps the southern hemisphere around the
other factor.  If we collapse the equator of $S^k$ to a point to get
$S^k\vee S^k$, then $\iota\circ \beta$ induces a map $S^k\vee S^k\to
S^n\vee S^n$ that sends each $S^k$ to one of the $S^n$'s by a map
homotopic to $\beta$.  We thus have
$$\iota\circ \beta=\iota_1\circ\beta+\iota_2\circ \beta$$
as desired. 

We can then generalize Theorems~\ref{thm:heisHigher} and \ref{thm:euclHigher} as follows:
\begin{thm}\label{thm:heisHilton}
  Let $n\ge 2$ and $n+2\le k$.  Let $\alpha:S^n\to \H^n$ be a
  Lipschitz map and let $\beta:S^k\to S^n$ be a Lipschitz map
  such that \eqref{eq:HH} holds.  (For example, if $\beta$ is a suspension.)  Then $\alpha \circ \beta:S^k\to \H^n$ can be
  extended to a Lipschitz map $r:D^{k+1}\to \H^n$.
\end{thm}
\begin{thm}\label{thm:euclHilton}
  Let $n\ge 2$ and $n+1\le k$.  Let $\beta:S^k\to S^n$ be a Lipschitz map
  such that \eqref{eq:HH} holds.  Then $\beta$ can be
  extended to a Lipschitz map $\gamma:D^{k+1}\to \R^{n+1}$ whose derivative
  has rank $\le n$ a.e.
\end{thm}
In particular, if $n\ge 2$ and $k<2n-1$, then for all $j$, we have
$q_j\ge 2$, so $h_j(\beta)\in \pi_k(S^{q_j(n-1)+1})=0$, and
\eqref{eq:HH} holds.  (Alternatively, one can note that $\beta$ is a
suspension by the Freudenthal suspension theorem \cite{Freuden}.)  Theorems~\ref{thm:heisHigher} and
\ref{thm:euclHigher} thus follow from Theorems~\ref{thm:heisHilton}
and \ref{thm:euclHilton}.

The proof of Theorem~\ref{thm:heisHilton} is based on that of Theorem~\ref{thm:euclHilton}, so we will prove it first. 
\begin{proof}[{Proof of Thm.\ \ref{thm:euclHilton}}]
  Let $I^n=[0,1]^n$ be the unit $n$-cube.  It
  suffices to consider the case that $\beta:\partial
  I^{k+1}\to \partial I^{n+1}$ and construct an extension of $\beta$
  to all of $I^{k+1}$.  

  Our construction is based on a construction of Kaufman
  \cite{KaufmanSingular}.  We will construct a map on a cube by defining a
  Lipschitz map $h$ on a cube with holes in it, then filling each of
  the holes with a scaling of $h$.  Repeating this process defines a
  Lipschitz map on all of the cube except a Cantor set of measure zero, so we finish
  by extending the map to the Cantor set by continuity.

  Let $\epsilon>0$ be such that
  $(2\epsilon)^{-(k+1)}>\epsilon^{-(n+1)}$ and $\frac{1}{\epsilon}\in \N$.  Subdivide
  $I^{n+1}$ into a grid of $N=\epsilon^{-(n+1)}$ cubes of side length
  $\epsilon$ and let $J$ be the $n$-skeleton of this grid.  Number the
  subcubes $1,2,\dots,N$ and let $J_i$ be the $i$th subcube.

  Subdivide $I^{k+1}$ into $(2\epsilon)^{-(k+1)}$ cubes of side length
  $2\epsilon$ and choose $N$ of these subcubes, numbered $1,\dots, N$.
  For $i=1,\dots,N$, we let $K_i$ be a cube of side length $\epsilon$,
  centered at the center of the $i$th subcube and let
  $$K=I^{k+1}\smallsetminus \bigcup_{i=1}^{N} K_i.$$

  We will define a Lipschitz map $h:K\to J$ that sends the boundaries
  of cubes in $K$ to the boundaries of cubes in $J$.  Since the image
  is an $n$-complex in $\R^{n+1}$, the derivatives of $h$ will have
  rank $\le n$ a.e.  First, we define $h$ on $\partial K$.  The
  boundary of $K$ is $\partial I^{k+1} \cup \bigcup \partial K_i$; let
  $h=\beta$ on $\partial I^{k+1}$, and define $h$ on $\partial K_i$ as
  a scaling and translation $\beta_i$ of $\beta$ which sends $\partial
  K_i$ to $\partial J_i$.  So far, this definition is Lipschitz.

  Next, we extend $h$.  Choose basepoints $x\in \partial I^{k+1}$ and
  $x_i\in \partial K_i$ and a collection of non-intersecting curves
  $\lambda_i$ connecting $x$ to $x_i$.  We can give $K$ the structure
  of a CW-complex, with vertices $x, x_1,\dots, x_{N}$; edges
  $\lambda_i$; $k$-cells $\partial I^{k+1}, \partial
  K_1,\dots, \partial K_{N}$; and a single $(k+1)$-cell.  We have
  already defined $h$ on all of the vertices and $k$-cells, and since
  $J$ is connected, we can extend $h$ to the edges of $K$.  It only
  remains to extend it to the $(k+1)$-cell.
  
  Consider the map $g:S^k\to J$ coming from the boundary of the
  $(k+1)$-cell.  The complex $J$ is homotopy equivalent to $\vee^{N}
  S^n$, because $J$ and $\vee^{N} S^n$ are both homotopy equivalent to
  an $(n+1)$-ball with $N$ punctures.  Furthermore, if $\iota_i:\partial
  I^{n+1}\to \partial J_i$ is the scaling and translation that sends
  the boundary of the cube to the boundary of the $i$th subcube, we
  can choose the homotopy equivalence so that the inclusions into each
  factor of $\vee^{N} S^n$ are homotopic to the $\iota_i$'s.  

  Thus $\pi_n(J)=\Z^N$, with generating set $\{\iota_i\}_{i=1}^N$.
  Given a map $f:S^n\to J$, we can write $f$ as a linear combination
  of the $\iota_i$'s by considering $f$ as a map $S^n\to I^{n+1}$ and
  calculating the winding numbers of $f$.  For each $i$, let $z_i$ be
  the center of the subcube $J_i$ and let $w_i(f)$ be the winding
  number of $f$ with respect to $z_i$, i.e., the image of the
  generator of $H_n(S^n)$ in $H_n(\R^{n+1}\setminus \{z_i\})=\Z$.
  Then
  $$f=\sum_{i=1}^N w_i(f) \cdot \iota_i.$$

  If $\iota:\partial I^{n+1} \hookrightarrow J$ is the inclusion of
  the boundary of the entire unit cube, then $\iota$ winds once around
  each point in the interior of $I^{n+1}$.  Therefore, $w_i(\iota)=1$
  for all $i$, and $\iota$ is homotopic to $\sum_{i=1}^{N} \iota_i$.
  We can write
  \begin{align*}
    g &=\iota\circ  \beta-\sum_{i=1}^{N}\beta_i\\
    &=\biggl(\sum_{i=1}^{N} \iota_i\biggr)\circ
    \beta-\sum_{i=1}^{N}(\iota_i \circ \beta)
  \end{align*}
  where the above equation is taken in $\pi_k(J)$.  By hypothesis,
  this is null-homotopic, so $h$ can be extended
  continuously to a map $K\to J$.  In fact, by a smoothing argument, this
  extension can be made Lipschitz.

  We can construct an increasing sequence 
  $$X_0=K\subset X_1\subset X_2\subset\dots$$
  by gluing together scaled copies of $K$ as follows.  Let $X_0=K$.  To construct
  $X_{i+1}$ from $X_i$, we glue a copy of $K$, scaled by
  $\epsilon^{i+1}$, to each of the cubical holes of $X_i$.  This
  replaces a hole of side length $\epsilon^{i+1}$ by $N$ holes of side
  length $\epsilon^{i+2}$, so for each $i$, $X_i$ is the complement of
  $N^{i+1}$ cubes of side length $\epsilon^{i+1}$ in $I^{n+1}$.  The
  union $\bigcup_{i=0}^\infty X_i$ is the complement in $I^{n+1}$ of a Cantor set
  of measure zero.

  For each $i$, we will construct a map $r_i:X_i\to I^{n+1}$
  such that,
  \begin{itemize}
  \item $r_{i+1}$ extends $r_i$,
  \item $\Lip r_i\le \Lip h$,
  \item the derivative of $r_i$ has rank $\le n$ a.e., and
  \item the restriction of $r_i$ to the boundary of one of the
  holes of $X_i$ is a copy of $\beta$ scaled by $\epsilon^{i+1}$.
  \end{itemize}
  Let $r_0=h$ on $X_0$.  This satisfies all the above conditions.  For
  any $i$, we construct $X_{i+1}$ from $X_i$ by gluing copies of $K$
  to holes in $X_i$.  On each new copy of $K$, we let $r_{i+1}$ be a
  copy of $h$ scaled by $\epsilon^{i+1}$.  This agrees with $r_i$ on
  the boundary of the copy of $K$, and since we scaled the domain and
  the range by the same factor, we still have $\Lip r_{i+1}= \Lip h$.

  The direct limit of the $r_i$ is a map 
  $$r:\bigcup_{i=0}^\infty X_i\to I^{n+1}$$
  defined on the complement of a Cantor set in $I^{k+1}$ with $\Lip
  r\le \Lip h$.  If we extend $r$ to all of $I^{k+1}$ by continuity,
  we get a Lipschitz extension of $\beta$ whose derivative has rank
  $\le n$ a.e.
\end{proof}

The construction in the Heisenberg group is similar.  Note that
because fillings of $n$-spheres in the $n$-th Heisenberg group have
Hausdorff dimension at least $n+2$ \cite[3.1.A]{GromovCC}, we need
$k\ge n+2$ rather than $k\ge n+1$.  (Gromov showed that sets of
topological dimension $\ge n+1$ must have Hausdorff dimension at least
$n+2$, and a filling of an embedded $n$-sphere must have topological
dimension $>n$.)  We will also need the following theorem
about low-dimensional Lipschitz extensions to Heisenberg groups:
\begin{theorem}[{\cite[3.5.D]{GromovCC}, \cite{Wenger-Young-Lip-ext}}]\label{thm:heisExtend}
  For any $n$, there is a $c>0$ such that if $X$ is a cube complex of
  dimension $\le n$, and if $f_0:X^{(0)}\to \H^n$ is a Lipschitz map
  defined on the vertices of $X$, then there is a Lipschitz extension
  $f:X\to \H^n$ of $f_0$ such that $\Lip f\le c \Lip f_0$.
\end{theorem}
The proof of this theorem involves repeatedly extending a map defined
on the boundary $\partial I^k$ of a unit $k$-cube to the entire
$k$-cube.  One first extends $f_0$ to a Lipschitz map on the
1-skeleton of $X$, then inductively to higher-dimensional skeleta.  It
is important that the original map $f_0$ is defined on every vertex of
$X$; without this condition, $c$ would have to depend on the complex
$X$ as well.

\begin{proof}[{Proof of Thm.\ \ref{thm:heisHilton}}]
  As before, we may replace $S^k$ and $S^n$ with $\partial I^{k+1}$
  and $\partial I^{n+1}$.  We will start by constructing an $n$-complex
  $J$ which is homotopy equivalent to a wedge of spheres, a subset $K$
  of $I^{k+1}$, and a Lipschitz map $h:K\to J$.  The main difference
  between this construction and the previous one is that $J$ will be a
  complex equipped with a map $\bar{\alpha}:J\to \H^n$ rather than a
  subset of $\H^n$.

  For any $\epsilon>0$ such that $1/\epsilon\in \N$, consider the
  complex $((\partial I^{n+1})\times [0,1/\epsilon])\cup
  (I^{n+1}\times \{0\})$.  This has $2n+2$ faces of the form
  $I^{n}\times [0,1/\epsilon]$ and one of the form $I^{n+1}$ and we
  can tile it with a total of
  $$N(\epsilon)=(2n+2)\epsilon^{-(n+2)}+\epsilon^{-(n+1)}$$
  cubes of side length $\epsilon$.  Let $J(\epsilon)$ be the $n$-skeleton of
  this tiling.  We identify $\partial I^{n+1}$ with $\partial
  I^{n+1}\times\{1/\epsilon\}\subset J(\epsilon)$.

  We claim that there is some $c>0$ such that if $f:\partial
  I^{n+1}\to \H^n$ is a Lipschitz map, then for any $\epsilon>0$,
  there is a Lipschitz extension $\bar{f}:J(\epsilon)\to \H^n$ with
  Lipschitz constant $\Lip{\bar{f}}\le c \Lip{f}$.  Recall that there
  is a family of dilations $s_t:\H^n\to \H^n$ such that $s_t(0)=0$ for
  all $t$ and $d(s_t(u),s_t(v))=td(u,v)$.  After composing
  $f$ with a dilation of $\H^n$ and translating it so that its image
  contains the identity, we may assume that $\Lip(f)=1$ and that
  $f(\partial I^{n+1})$ is contained in the ball $B\subset \H^n$
  around the identity of radius $n+1$.  Define $\bar{f}:J(\epsilon)\to
  \H^n$ on the vertices of $J(\epsilon)$ as
  $$\bar{f}(v,t)=\begin{cases}
    s_{\epsilon t}(f(v)) & t>0 \\
    1 & t=0
  \end{cases}$$ where $v\in I^{n+1}, t\in [0,1/\epsilon]$ and where
  $s_t:\H^n\to \H^n$ is dilation by a factor of $t$.  We claim that
  this is Lipschitz on the vertices with Lipschitz constant independent of
  $\epsilon$; then, by Theorem~\ref{thm:heisExtend}, we can extend it to a
  $c$-Lipschitz map on all of $J(\epsilon)$.
  
  It suffices to show that the distance between the
  images of any two adjacent vertices is $O(\epsilon)$, with implicit
  constant depending only on $n$.  If the two vertices are $(v,0)$ and
  $(v',0)$, the map sends both of them to the identity.  If $v$ is
  adjacent to $v'$ in $\partial I^{n+1}$ and $t\in (0,1/\epsilon]$,
  then
  $$d(\bar{f}(v,t),\bar{f}(v',t))=d(s_{\epsilon t}(f(v)), s_{\epsilon
    t}(f(v'))) \le d(v,v')=\epsilon.$$
  
  If the vertices are of the form $(v,t), (v,t')$, with
  $|t-t'|=\epsilon$, let $f(v)=(x,y,z)$ for
  $x,y\in \R^n$ and $z\in \R$.  
  On any compact set, 
  $$d_{\H^n}((a,b,c),(a',b',c'))=O(\sqrt{\|a-a'\|+\|b-b'\|+\|c-c'\|}),$$
  and since $\bar{f}(v,t),\bar{f}(v,t')\in B$, 
  $$d_{\H^{n}}(\bar{f}(v,t),\bar{f}(v,t'))=O(\sqrt{\epsilon|t-t'|
  \|x\|+ \epsilon|t-t'|\|y\|+ \epsilon^2|t^2-t'^2|\|z\|})=O(\epsilon)$$ 
  as desired.
  
  Choose $\epsilon>0$ such that 
  $$N(\epsilon)\le \left\lfloor\frac{1}{2 c \epsilon}\right\rfloor^{k+1}$$
  (this is possible because $k\ge n+2$) and let $J=J(\epsilon)$,
  $N=N(\epsilon)$.  Label the cubes of $J$ by $1,\dots,N$.
  
  Next, we construct $K$.  We can subdivide $I^{k+1}$ into at least
  $N$ subcubes, each with side length at least $2c\epsilon$.  Number
  $N$ of these subcubes $1,\dots, N$, and for each $i$, let $K_i$ be a
  cube of side length $c\epsilon$ centered at the center of the $i$th
  subcube.  Let $K=I^{k+1}\smallsetminus \bigcup_{i=1}^N K_i$.  As in
  the proof of Theorem\ \ref{thm:euclHilton}, construct a Lipschitz
  map $h:K\to J$ such that for each $i$, $\partial K_i$ is mapped to
  $\partial J_i$ by a scaling of $\beta$.

  Define $X_0=K\subset X_1\subset\dots$ as before, so that $X_i$
  consists of $I^{k+1}$ with $N^{i+1}$ cubical holes of side length
  $(c\epsilon)^{i+1}$.  Let $Y_0=J$.  This consists of $N$ cubical
  holes of side length $\epsilon$.  For each $i$, we let $Y_{i+1}$ be
  $Y_i$ with a scaled copy of $J$ glued to each cubical hole, so that
  for each $i$, $Y_i$ is an $n$-complex consisting of the boundaries
  of $N^{i+1}$ cubes of side length $\epsilon^{i+1}$.  We construct
  maps $\gamma_i:X_i\to Y_i$ inductively.  We start by letting
  $\gamma_0=h$.  By induction, if $C$ is the boundary of one of the
  holes in $X_i$, $\gamma_i$ sends $C$ to the boundary $D$ of a hole
  in $Y_i$.  To construct $X_{i+1}$ from $X_i$, we glue a scaled
  copy of $K$ to $C$, and to construct $Y_{i+1}$ from $Y_i$, we
  glue a scaled copy of $J$ to $D$.  We extend $\gamma_i$ to
  $\gamma_{i+1}$ by sending each scaled copy of $K$ to the
  corresponding scaled copy of $J$ by a scaled copy of $h$.  Note that
  since the scaling factors in the construction of $X_i$ and $Y_i$ are
  different, the Lipschitz constant of $\gamma_i$ varies from point to point;
  if $Z$ is a connected component of $X_i\setminus X_{i-1}$, then
  $$\Lip \gamma_i|_Z\le  c^{-i}\Lip h.$$

  Finally, we construct maps $\sigma_i:Y_i\to \H^n$.  We proceed inductively.
  As noted above, any Lipschitz map $f:\partial I^{n+1}\to \H^n$ can
  be extended to a Lipschitz map $\bar{f}:J\to \H^n$ with
  $\Lip \bar{f}\le c\Lip f$.  We will construct a sequence of maps
  $\sigma_i:Y_i\to \H^n$ with $\Lip\sigma_i \le c^{i+1} \Lip \alpha$.
  Let $\sigma_0=\bar\alpha:Y_0\to \H^n$; we have $\Lip\sigma_0 \le c
  \Lip \alpha$.  For each $i$, the complex $Y_{i+1}$ consists of $Y_i$
  with $N^{i+1}$ copies of $J$ glued on, so we can extend $\sigma_i$
  to $Y_{i+1}$ by constructing an extension over every copy of $J$.
  By induction, $\Lip\sigma_i \le c^{i+1} \Lip \alpha$, so
  $\Lip\sigma_{i+1} \le c^{i+2} \Lip \alpha$ as desired.

  Let $r_i=\sigma_i\circ \gamma_i$.  If $Z$ is a connected
  component of $X_i\setminus X_{i-1}$, then
  $$\Lip r_i|_Z \le  c^{-i}(\Lip h) c^{i+1} (\Lip \alpha)\le c (\Lip h)( \Lip \alpha),$$
  so the $r_i$ are uniformly Lipschitz.  Their direct limit is a
  Lipschitz map from the complement of a Cantor set to $\H^n$ which
  extends $\alpha\circ \beta$.  Extending this to all of $I^{k+1}$ by
  continuity, we get the desired $r$.
\end{proof}

\section{Factoring through trees}\label{section:factor-tree}
The aim of this section is to prove Theorem~\ref{thm:factor-tree}. For
this, let $X$ and $Y$ be metric spaces as in the statement of the
theorem and let $f: X\to Y$ be a Lipschitz map. Roughly, the idea of
the proof is to pull back the metric of $Y$ by $f$ and show that the
resulting metric space is a tree. 

Define a pseudo-metric on $X$ by
\begin{equation*}
 d_f(x, x'):= \inf\{\length(f\circ c): \text{ $c$ Lipschitz curve from $x$ to $x'$}\}
\end{equation*}
and note that 
\begin{equation}\label{eq:d_f}
 d_f(x, x')\leq C\Lip(f) d(x,x')
\end{equation}
for all $x,x'\in X$, where $C$ is the quasi-convexity constant of
$X$. Let $Z$ be the quotient space $Z:= X/_\sim$ by the equivalence
relation given by $x\sim x'$ if and only if $d_f(x, x')=0$. Endow $Z$
with the metric $$d_Z([x], [x']):= d_f(x, x'),$$ where $[x]$ denotes
the equivalence class of $x$, and define maps $\psi: X\to Z$ and
$\varphi: Z\to Y$ by $\psi(x):= [x]$ and $\varphi([x]):= f(x)$, so
that $f = \varphi\circ \psi$.  It follows from
\eqref{eq:d_f} that $\psi$ is $C\Lip(f)$-Lipschitz. Since $$d(f(x),
f(x'))\leq \length(f\circ c)$$ for all Lipschitz curves $c$ from $x$
to $x'$ we moreover infer that $\varphi$ is well-defined and
$1$-Lipschitz.

Furthermore, for any Lipschitz curve $\gamma:[0,1]\to X$,
\begin{equation}\label{eq:lengthSpace}
\length(\psi \circ \gamma)\leq   \length (f\circ \gamma).
\end{equation}
In particular, $Z$ is a length space.

For any Lipschitz closed curve $\gamma=(\gamma_1,\gamma_2):S^1\to \R^2$, let
\begin{equation*}
  A(\gamma)=\int_{S^1} \gamma_1(t)\cdot \gamma_2'(t)\;dt
\end{equation*}
be the signed area of $\gamma$.  We will show:
\begin{lemma}\label{lem:triv-curve-integral}
  For every Lipschitz curve $\alpha: S^1 \to Z$ and every Lipschitz
  map $\pi=(\pi_1,\pi_2): Z\to\R^2$, we have
  \begin{equation}\label{eq:triv-curve-integral}
    A(\pi \circ \alpha)=0.
  \end{equation}
\end{lemma}
Proposition~3.1 of \cite{Wenger-tree} implies that any geodesic
metric space satisfying Lemma~\ref{lem:triv-curve-integral} is in fact
a metric tree.
\begin{proof}
 We first show that $A(\pi \circ \psi \circ \beta)=0$ for every Lipschitz curve $\beta: S^1\to X$.
 Fix $\beta$ and $\pi$. By the hypotheses on $X$, there is some Lipschitz map
 $\varrho:D^2\to X$ which extends $\beta$.  We first claim that the metric derivative
 $\md(\psi\circ\varrho)_z$ is degenerate
 for almost every $z\in D^2$.
For this, let $0<\varepsilon<1/2$ and let $\{v_1,\dots,v_k\}\subset S^1$ be a finite $\frac{\varepsilon}{\lambda}$-dense subset, where $\lambda=\Lip(f\circ\varrho)$. For $i\in\{1,2,\dots, k\}$ define 
\begin{equation*}
 A_i:= \{z\in D^2: \text{ $\md(f\circ\varrho)_z$ exists, is a seminorm, and $\md(f\circ\varrho)_z(v_i)\leq \varepsilon$}\}.
\end{equation*} 
It is not difficult to show that 
\begin{equation}\label{eq:tutti-quanti}
\left|D^2\backslash\bigcup_{i=1}^k A_i\right| = 0,
\end{equation}
where $|\cdot|$ denotes the Lebesgue measure on $\R^2$.
 Indeed, for almost every $z\in D^2$ the metric derivative $\md(f\circ\varrho)_z$ exists and is a seminorm. Since $Y$ is purely $2$-unrectifiable it follows from the area formula \cite[Thm.\ 7]{Kirchheim-rectifiable} that $\md(f\circ\varrho)_z$ is degenerate for almost every $z\in D^2$. Thus, given such $z$, there exists $v\in S^1$ such that $\md(f\circ\varrho)_z(v)=0$. Choose $i$ such that $|v-v_i|\leq \varepsilon/\lambda$. It follows that $$\md(f\circ\varrho)_z(v_i)\leq \md(f\circ\varrho)_z(v_i - v) \leq \varepsilon.$$ This proves \eqref{eq:tutti-quanti}. Now, fix $i\in\{1, 2, \dots, k\}$ and let $z\in A_i$ be a Lebesgue density point. Let $r_0>0$ be such that $B(z, 2r_0)\subset D^2$ and  
 \begin{equation}\label{eq:Lebesgue-density}
  \frac{|B(z,r)\backslash A_i|}{|B(z,r)|} \leq 100^{-1}\varepsilon^2
 \end{equation}
  for all $r\in(0,2r_0)$. Let $v_i^\perp\in S^1$ be a vector orthogonal to $v_i$ and let $r\in(0,r_0)$. For each $s\in (0, \varepsilon r)$ let $C_s$ denote the set $C_s:= \{t\in[0,r]: z+s v_i^\perp + t v_i\not\in A_i\}$. It follows from Fubini's theorem and \eqref{eq:Lebesgue-density} that there exists a subset $\Omega\subset(0, \varepsilon r)$ of strictly positive measure such that $\haus^1(C_s)\leq \varepsilon r$ for every $s\in\Omega$. Let $s\in\Omega$ and denote by $\gamma$ the piecewise affine curve in $\R^2$ connecting $z$ with $z+rv_i$ via $z+sv_i^\perp$ and $z+sv_i^\perp+rv_i$. It now follows that
  \begin{equation*}
    \length(f\circ\varrho\circ\gamma) \leq 2s\lambda + \int_0^r\md(f\circ\varrho)_{z+sv_i^\perp+tv_i}(v_i)dt
     \leq 2s\lambda + \varepsilon r + \lambda |C_s|
     \leq (3\lambda+1)\varepsilon r
  \end{equation*}
 and hence that for every $r\in(0,r_0)$
  \begin{equation*}
   \frac{1}{r} d_Z(\psi\circ\varrho(z), \psi\circ\varrho(z+rv_i)) = \frac{1}{r}d_f(\varrho(z), \varrho(z+rv_i)) \leq (3\lambda + 1)\varepsilon.
  \end{equation*}
  In particular, if $\md(\psi\circ\varrho)$ exists at $z$ and is a seminorm then $\md(\psi\circ\varrho)_z(v_i)\leq (3\lambda+1)\varepsilon$. Since $\varepsilon>0$ was arbitrary this shows that $\md(\psi\circ\varrho)_z$ is degenerate for almost all $z\in D^2$, as claimed. It now follows that $\det\left(\nabla(\pi\circ\psi\circ\varrho)\right) =0$ almost everywhere on $D^2$ and hence, by a smoothing argument and Stokes' theorem, that
 $$A(\psi\circ\beta) = \int_{D^2} \det\left(\nabla(\pi\circ\psi\circ\varrho)\right) = 0.$$

Now, let $\alpha: S^1 \to Z$ be a Lipschitz curve.  We identify $S^1$ with the interval $[0,2\pi]$ with its endpoints glued together.  We will construct a sequence of Lipschitz curves $\beta_n: S^1\to X$, $n=1, 2, \dots,$ such that $\psi\circ\beta_n$ converges uniformly to $\alpha$ and such that $\length(\psi\circ\beta_n)\leq 2\Lip(\alpha)$ for every $n\in\N$.  Namely, for every $t\in S^1$, let $x_t\in X$ be a representative of the equivalence class $\alpha(t)$.  Then for every $t,u\in S^1$ with $t<u$, we have $d_f(x_t,x_u)\le \Lip(\alpha) |t-u|$, so there is a Lipschitz curve $\gamma_{t,u}:[0,2\pi]\to X$ from $x_t$ to $x_u$ such that 
$\length (f\circ \gamma_{t,u})\le 2 \Lip(\alpha) |t-u|.$  By \eqref{eq:lengthSpace}, we have
$$\length (\psi \circ \gamma_{t,u}) \leq \length (f\circ \gamma_{t,u})\le 2 \Lip(\alpha) |t-u|.$$
Let $\beta_n$ be the concatenation $\gamma_{0,2\pi/n}\dots \gamma_{2\pi(n-1)/n,2\pi}.$  Since $\psi\circ \gamma_{t,u}$ stays within distance $2 \Lip(\alpha) |t-u|$ of $\alpha(t)$, we have $\psi\circ\beta_n\to \alpha$ uniformly.  In addition, since $\length(\psi\circ\beta_n)$ is uniformly bounded, we have 
$$A(\pi\circ\alpha)=\lim_{n\to \infty} A(\pi\circ \psi\circ \beta_n)$$ 
for every Lipschitz map $\pi:Z\to\R^2$.  Since $A(\pi\circ \psi\circ
\beta_n)=0$ for all $n$, this proves \eqref{eq:triv-curve-integral}.
\end{proof}

Furthermore, any curve satisfying \eqref{eq:triv-curve-integral} for
every $\pi$ is non-injective.
\begin{lemma}\label{lem:inj-Lip-loop}
 Let $(Z',d)$ be a metric space and $\gamma: S^1\to Z'$ be an injective Lipschitz curve. Then there exists a Lipschitz map $\pi=(\pi_1,\pi_2): Z'\to\R^2$ such that $$A(\pi\circ\gamma) \not=0.$$
\end{lemma}
 \begin{proof}
  We identify $S^1$ with $[0,2\pi]$ with the endpoints identified. Let $0<a<b<2\pi$ and let $0< \varepsilon< (2\Lip(\gamma))^{-1} d(\gamma(b), \gamma(a))$ be so small that $0<a-\varepsilon< b+\varepsilon<2\pi$.  Since $\gamma$ is a homeomorphism onto its image there exists $U\subset Z'$ open with $$\gamma((a-\varepsilon, b+\varepsilon)) = U\cap \gamma(S^1).$$ Let $\delta>0$ be so small that the open $\delta$-neighborhood of $\gamma([a,b])$ is contained in $U$.
  Define $\pi_1: Z'\to\R$ by $$\pi_1(z):= \max\left\{0, 1 - \delta^{-1} d(z, \gamma([a,b]))\right\}.$$
  Clearly, $\pi_1$ is $\delta^{-1}$-Lipschitz with $\pi_1=1$ on $\gamma([a,b])$ and $\pi_1=0$ on $U^c$. Define furthermore a $1$-Lipschitz function $\pi_2:Z'\to\R$ by $\pi_2(z):= d(z, \gamma(a))$. Since $(\pi_1\circ\gamma)(t) = 1$ for all $t\in[a,b]$ it follows that
  $$\int_a^b(\pi_1\circ\gamma)(t)\cdot (\pi_2\circ\gamma)'(t) dt = \pi_2(\gamma(b)) - \pi_2(\gamma(a)) = d(\gamma(b), \gamma(a)).$$ From this and the fact that $\gamma(t)\not\in U$ for all $t\in(a-\varepsilon, b+\varepsilon)$ we finally obtain 
 \begin{equation*}
   \begin{split}
    \int_0^{2\pi} (\pi_1\circ\gamma)(t)\cdot (\pi_2\circ\gamma)'(t)dt &= \int_{a-\varepsilon}^{b+\varepsilon} (\pi_1\circ\gamma)(t)\cdot (\pi_2\circ\gamma)'(t)dt\\
  &   \geq d(\gamma(b), \gamma(a)) - 2\varepsilon \Lip(\gamma)\\
    &>0.
   \end{split}
  \end{equation*} 
  This completes the proof of the lemma.
 \end{proof}

 It remains to show that $Z$ is geodesic. Since $Z$ is a length space,
 there is an injective Lipschitz curve $\alpha:[0,1]\to Z$ connecting
 $z$ to $z'$.  If $\alpha'$ is another such curve, we claim 
 that $\alpha([0,1]) = \alpha'([0,1]).$
Indeed, if this were not true we would find subintervals $(s_1,s_2)$ and $(t_1,t_2)$ of $(0,1)$ such that $$\alpha((s_1, s_2)) \cap \alpha'((t_1,t_2)) = \emptyset$$ and such that the endpoints of $\alpha|_{[s_1,s_2]}$ agree with those of $\alpha'|_{[t_1,t_2]}$. 
We would thus obtain an injective Lipschitz curve $\gamma: S^1\to
Z$. This is impossible by Lemma~\ref{lem:triv-curve-integral} and
Lemma~\ref{lem:inj-Lip-loop}. Consequently, there is a unique simple
path between any two points, and $Z$ is geodesic.  Thus, by \cite[Proposition 3.1]{Wenger-tree}, $Z$ is a metric tree. This concludes the proof of Theorem~\ref{thm:factor-tree}.

\bibliography{lh}
\end{document}